\documentclass[11pt]{article}
\usepackage{graphicx}
\usepackage{amssymb}
\usepackage{amsmath}
\usepackage{amsthm}
\usepackage{latexsym}
\usepackage[english]{babel}

\newtheorem{defi}{Definition}[section]
\newtheorem{teo}[defi]{Theorem}
\newtheorem{pro}[defi]{Proposition}

\newtheorem{lemma}[defi]{Lemma}
\newtheorem{cor}[defi]{Corollary}

\begin{document}


\title{On a Fourth Order Lichnerowicz Type Equation Involving The
Paneitz-Branson Operator.}

\author{Ali Maalaoui$^{1}$} 
\addtocounter{footnote}{1}
\footnotetext{Department of Mathematics,
Rutgers University - Hill Center for the Mathematical Sciences
110 Frelinghuysen Rd., Piscataway 08854-8019 NJ, USA. E-mail address:
{\tt{maalaoui@math.rutgers.edu}}}

\maketitle

{\noindent\bf Abstract} In this paper, we study some fourth order singular critical equations of Lichnerowicz type involving the Paneitz-Branson operator, and we prove existence and non existence results under given assumptions.

\section{Introduction}

During the last years there have been effective studies of conformal
operators and their relative invariants due to their application in geometry
or mathematical physics. For instance the Yamabe problem played an essential
role in the evolution of the analytical and geometrical tools also it was
with crucial importance for the study of the Einstein-Hilbert functional
without forgetting the input in relativity for the study of the conformal
Einstein constraint equations (see \cite{Cho3},\cite{Cho2}). And In 1983,
Paneitz \cite{Pan} introduced a conformally fourth order operator defined on
4-dimensional Riemannian manifolds Branson \cite{Bra1} generalized the
definition to $n-$dimensional Riemannian manifolds, $n\geq 5$. He introduced
another geometric quantity that defines another conformal invariant for $n>3$
that is the $Q-curvature$ that behaves in a very similar way to the scalar
curvature. And its variation after a conformal change involves a fourth
order operator. That is the Paneitz-Branson Operator. One can think about
the $Q-curvature$ and the Branson-Paneitz operator like the scalar
curvature and the conformal laplacian. There was a lot of published work
concerning prescribing the $Q-$curvature where one can notice that the
conditions that we get are similar to the scalar curvature one modulo some
technical assumptions (see \cite{Dja2},\cite{Dja3},\cite{Ben}). One of the
issues that we meet while dealing with this operator, is the fact that there
is no maximum principle, thus getting good and effective estimate is not as
easy as for the conformal Laplacian. Many authors have studied the
positivity and coercivity of the Paneitz, one can consult \cite{Yang} or 
\cite{Heb2} for example. As interaction with mathematical physics we can see
the work of Choquet-Bruhat in \cite{Cho3} with the Conformal Laplacian,
where they study the scalar field equation that leads to a Lichnerowicz type
semi-linear PDE. In this work we attempt to study another action functional
as a proposal for a relativestic model since it is conformally invariant and
we will see a scalar-field perturbation of it. The study of such functional
leads to the resolution of a Lichnerowicz type equation but it is a fourth
order one, with the Paneitz operator as a differential part. So in this
paper we will investigate the existence of positive solutions under some
assumption that we will mention later to that equation and also we will give
a non-existence result.

\section{Preliminaries and Motivations }

Let $\left( M,g\right) $ be a $n-$dimensional closed compact manifold with $%
n\geq 3,$ recall that if $R_{g}$ is the scalar curvature then under the
conformal change $\widetilde{g}=u^{\frac{4}{n-2}}g$, one gets the following
relation relation the new curvature with the old one : 
\begin{equation}
-\Delta _{g}u+\frac{(n-2)}{4(n-1)}R_{g}u=R_{\widetilde{g}}u^{\frac{n+2}{n-2}%
},  \label{scal}
\end{equation}%
Let $-L_{g}u=-\Delta _{g}u+\frac{(n-2)}{4(n-1)}R_{g}u$ this operator is
called the conformal Laplacian to see more of its property one could check 
\cite{Lee}. Similarly if we consider the following quantity which is the
Branson $Q-curvature$ introduced in \cite{Bra1}, defined by 
\begin{eqnarray*}
Q &:&=\frac{n^{2}-4}{8n(n-1)^{2}}R^{2}-\frac{2}{(n-1)^{2}}\left\vert Ric-%
\frac{R}{n}g\right\vert ^{2}+\frac{1}{2(n-1)}\Delta R \\
&=&c_{n}R^{2}+d_{n}\left\vert Ric\right\vert ^{2}+\frac{1}{2(n-1)}\Delta R.
\end{eqnarray*}%
Then after a conformal change $\widetilde{g}=u^{\frac{4}{n-4}}g$ of the
metric, one gets 
\begin{equation}
Q_{\widetilde{g}}u^{\frac{n+4}{n-4}}=Pu  \label{PerQ}
\end{equation}%
where

\begin{equation*}
Pu:=\Delta _{g}^{2}u-div\left( \left( \frac{\left( n-2\right) ^{2}+4}{%
2\left( n-2\right) (n-1)}Rg-\frac{4}{n-2}Ric\right) du\right) +\frac{n-4}{2}%
Qu.
\end{equation*}%
We will set $P_{0}$ its differential part, that is 
\begin{equation*}
P_{0}u=\Delta _{g}^{2}u-div\left( \left( \frac{\left( n-2\right)
^{2}+4}{2\left( n-2\right) (n-1)}Rg-\frac{4}{n-2}Ric\right) du\right) ,
\end{equation*}%
One can see that if $g$ is an Einstein metric then $P$ is with constant
coefficient, 
\begin{equation*}
Pu=\Delta _{g}^{2}u+\frac{n^{2}-2n-4}{2n(n-1)}R\Delta _{g}u+\frac{\left(
n-4\right) (n^{2}-4)}{16n(n-1)^{2}}R^{2}u
\end{equation*}%
and satisfies the maximum principle since it can be written as a product of
two second order operator satisfying the maximum principle.

Remark that the natural space to work on for the prescribed scalar
curvature problem is $H^{1}\left( M\right) $ and thus (\ref{scal}) is a
critical semi-linear problem. Also the natural space for prescribing the
Branson $Q$-curvature is $H^{2}\left( M\right) $ and again (\ref{PerQ}) is a
critical problem since we are in the borderline of the Sobolev embeddings.

There was an extensive work concerning (\ref{scal}) to find a metric with
constant scalar curvature, which is a critical point of the Einstein-Hilbert
functional 
\begin{equation*}
F_{R}:g\longmapsto \frac{\int_{M}R_{g}}{V_{g}^{\frac{n-4}{n}}}
\end{equation*}%
restricted to the conformal Class of a given metric. Same thing can be
applied to the functional$\ $ 
\begin{equation*}
F_{Q}:g\longmapsto \frac{\int_{M}Q_{g}}{V_{g}^{\frac{n-4}{n}}},
\end{equation*}

For instance one could check that the functional is Reimannian and Einstein
metrics are critical points of this functional.

\bigskip There have been many proposal in relativity to replace the
Hilbert-Einstein total curvature functional with a conformally invariant
functional, like for instance the case of Bach relativity where the
functional is replaced by $\int_{M}|C_{g}|^{2}dv_{g}$ and $C$ is the Weyl
tensor (see \cite{Bra2}). In this case we will consider another proposal
consisting of the total Paneitz-Branson curvature $F_{Q}$. Therefore one
can think about a scalar field perturbation of the previous one, that is 
\begin{equation*}
F_{\psi }\left( g\right) =\int_{M}Q_{g}-\left\vert \nabla _{g}\psi
\right\vert ^{2}-V(\psi )d\mu _{g},
\end{equation*}%
this functional was studied in \cite{Cho 4} for the case Hilbert-Einstein
functional under conformal change where the authors try to solve a conformal
constraint for the Einstein scalar field equation also in \cite{Heb1} where
the author studies the problem from a variational point of view.

Now if we take a closer look to this functional, one can see that if we
restrict it to the conformal class of $g$ one have 
\begin{equation*}
F_{\psi }\left( u^{\frac{4}{n-4}}g\right) =\frac{1}{a_{n}}%
\int_{M}uPu-a_{n}\left\vert \nabla _{g}\Psi \right\vert ^{2}u^{2}d\mu _{g}.
\end{equation*}%
Where $a_{n}=\frac{n-4}{4},$ therefore, the associated Euler-Lagrange
equation to this problem is 
\begin{equation*}
P_{g,\psi }u=Pu-a_{n}\left\vert \nabla _{g}\Psi \right\vert ^{2}u=\left( 
\widetilde{Q}-\left\vert \nabla _{\widetilde{g}}\Psi \right\vert ^{2}\right)
u^{\frac{n+4}{n-4}},
\end{equation*}%
that is 
\begin{equation*}
P_{0}u+a_{n}\left( Q-\left\vert \nabla _{g}\Psi \right\vert ^{2}\right)
u=\left( \widetilde{Q}-\left\vert \nabla _{\widetilde{g}}\Psi \right\vert
^{2}\right) u^{\frac{n+4}{n-4}}
\end{equation*}%
and the constant 
\begin{equation*}
\mathcal{P}\left[ g,\Psi \right] =\inf_{u>0,u\in C^{\infty }(M)}\frac{1}{%
a_{n}}\frac{\int_{M}uPu-a_{n}\left\vert \nabla _{g}\Psi \right\vert ^{2}u^{2}%
}{\left( \int_{M}u^{\frac{2n}{n-4}}\right) ^{\frac{n-4}{n}}},
\end{equation*}%
is a conformal invariant.

Let us recall the following result about the coercivity of the
Paneitz-Branson operator.

\begin{teo}[\cite{Yang}]
Let $(M,g)$ be a closed Riemannian manifold of dimension at least $6$. If
the Yamabe invariant of $g$ is non-negative, then with respect to any
conformal metric of positive scalar curvature $P_{0}$ has a non-negative
first eigenvalue and $kerP$ $_{0}=\{constant\}.$ The last statement also
holds in dimension five, provided we assume the Yamabe metric has positive $%
Q $-curvature.
\end{teo}

\begin{pro}
Under the assumptions of \ the previous theorem, then the sign of $\mathcal{P%
}\left[ g,\Psi \right] $ is the sign of the first eigenvalue of the Operator 
$P_{\widetilde{g},\Psi }$ for every $\widetilde{g}$ in the same conformal
class.
\end{pro}

\begin{proof}
Assume $\mathcal{P}\left[ g,\Psi \right] >0$, then if we take $\varphi _{1}$
the first eigenvalue as a test function, one gets 
\begin{equation*}
\lambda _{1}\left\Vert \varphi _{1}\right\Vert _{2}^{2}=E(\varphi _{1})\geq 
\mathcal{P}\left[ g,\Psi \right] \left\Vert \varphi _{1}\right\Vert
_{2^{\sharp }}^{2}
\end{equation*}%
thus $\lambda _{1}>0.$ Now if we assume that $\mathcal{P}\left[ g,\Psi %
\right] =0$ then using the same argument we get that $\lambda _{1}\geq 0$ ,
but if $\lambda _{1}>0$ then using Sobolev inequalities we get $\mathcal{P}%
\left[ g,\Psi \right] >0$ which is not the case, thus $\lambda _{1}=0.$

Now if $\mathcal{P}\left[ g,\Psi \right] <0$ then there exist a function $v,$
such that $E(v)<0$ thus we get $\lambda _{1}<0.$
\end{proof}

From now on we will assume that the Yamabe and the Paneitz invariants are
positive and $P$ is positive, therefore we guaranty that the operator $%
P_{g,\Psi }$ is coercive and satisfies the maximum principle and up to a
conformal change we can assume that $Q_{\psi }=Q-\left\vert \nabla \psi
\right\vert _{g}^{2}$ id positive on $M.$

And if we follow the procedure of the Authors in \cite{Cho 4} to find the
Einstein-scalar field conformal constraint equation one gets a a
Lichnerowicz type problem but of fourth order of the following form :%
\begin{equation*}
\left\{ 
\begin{array}{c}
P_{g,\Psi }u=\frac{A(x)}{u^{2^{\sharp }+1}}-B(x)u^{2^{\sharp }-1} \\ 
u>0%
\end{array}%
\right. .
\end{equation*}%
where $2^{\sharp }=\frac{2n}{n-4}$, $A$ and $B$ two smooth functions.
Therefore, the object of the rest of this paper is to investigate the
existence of positive solutions to problem of the following form :%
\begin{equation*}
\left\{ 
\begin{array}{c}
P_{g,\Psi }u=\frac{A(x)}{u^{p}}-B(x)u^{q} \\ 
u>0%
\end{array}%
\right. ,
\end{equation*}%
where $p>1$ and $1<q\leq 2^{\sharp }-1.$

\section{Existence Via heat flow }

Let $E$ be a Banach space with norm $\left\Vert -\right\Vert _{E}$. $E$ is
partially ordered by a closed cone $P$ $\subset X,$ and we assume that it
has non-empty interior $\overset{\circ }{P}.$ We define also $\overset{%
\bullet }{P}=P-\left\{ 0\right\} $. The element of $\overset{\bullet }{P}$
are called positive and element of $-$ $\overset{\bullet }{P}$ are called
negative.

Now, if $u,v\in E$ we will use the following notations to distinguish how
they are comparable :

\begin{equation*}
\left\{ 
\begin{array}{c}
u\leq v\text{ if }v-u\in P \\ 
u<v\text{ if }v-u\in \overset{\bullet }{P} \\ 
u\ll v\text{ if }v-u\in \overset{\circ }{P}%
\end{array}%
\right. .
\end{equation*}%
A map $f,$ we set $D(f)\subset E$ its domain. Now a map $f:D(f)%
\longrightarrow E$ is said order preserving, if for every $u,v\in D(f)$ such
that $u\leq v$ then $f(u)\leq f(v).$

And we say that 
\begin{equation*}
\lim_{x\longrightarrow +\infty }f(x)=+\infty
\end{equation*}%
if for every $u\in P$ there exist $x\in P$ such that $f(v)\geq u,$ for every 
$v\geq x.$

And finally we define the set 
\begin{equation*}
\left[ u,v\right] =\left\{ w\in E;u\leq w\leq v\right\}
\end{equation*}%
and sometimes if needed for a set $D\subset E,$ 
\begin{equation*}
\left[ u,v\right] _{D}=\left\{ w\in D;u\leq w\leq v\right\} .
\end{equation*}

An OBS is said normal if there exist $\delta >0$ such that for every $u\leq
v $ in $E$, 
\begin{equation*}
\left\Vert u\right\Vert _{E}\leq \delta \left\Vert v\right\Vert _{E}.
\end{equation*}

\begin{teo}[Krein-Rutman]
Let $E$ be a total ordered Banach space and $T:E\longrightarrow E$ a compact
order preserving linear operator, then $r(T)$ is an eigenvalue with
eigenvector $u\in \overset{\bullet }{P}$ and if in adition we assume that $T$
is strongly order preserving (That is $Tu>>0$ if $u>0$) then $r(T)>0$ and is
a simple eigenvalue with positive eingen vector.
\end{teo}

Let us consider the following problem

\begin{equation}
\left\{ 
\begin{array}{c}
\frac{d}{dt}u+Au=F(u) \\ 
u(0)=u_{0}%
\end{array}%
\right.  \label{Prob H}
\end{equation}%
where $F:\overset{\circ }{P}\longrightarrow X$ $\ $is a $C^{1}$map , $A$ is
a densely defined compact resolvent positive operator and $u_{0}\in \overset{%
}{\overset{o}{P}\cap D(A)}.$

\begin{teo}
Assume that 
\begin{equation*}
\lim_{x\longrightarrow +\infty }F(x)=-\infty
\end{equation*}%
and 
\begin{equation*}
\lim_{x\longrightarrow 0^{+}}F(x)=+\infty
\end{equation*}%
and for every bounded set $K$ there exist a a constant $\lambda $ such that $%
F+\lambda I$ is order preserving in $K,$ then the problem admits a positive
solution.\newline
\end{teo}

In the applications we can know more about the solution and we will deal
with that further in this paper.

\begin{proof}
First remark that there exist $u_{1}$ and $u_{2}$ such that 
\begin{eqnarray*}
\frac{d}{dt}u_{1}+Au_{1} &\leq &F(u_{1}) \\
u_{1}(0) &\leq &u_{0}
\end{eqnarray*}%
and 
\begin{eqnarray*}
\frac{d}{dt}u_{2}+Au_{2} &\geq &F(u_{2}) \\
u_{2}(0) &\geq &u_{0}
\end{eqnarray*}

In fact $u_{1}$ and $u_{2}$ could be chosen of the form $se$ where $e\in 
\overset{\circ }{P}$ and $s>0.$

set 
\begin{equation*}
K=\left\{ u_{1}\leq u\leq u_{2}\right\}
\end{equation*}%
then $K$ is a bounded set, so there exist $\lambda >0$ so that $F+\lambda I$
is order preserving on $K$ so let $\widetilde{A}$ and $\widetilde{F}$ denote
respectively $A+\lambda I$ and $F+\lambda I.$

Now let us construct the following sequence : $u^{1}$ being the unique
solution of 
\begin{equation*}
\frac{d}{dt}u+\widetilde{A}u=\widetilde{F}(u_{1})
\end{equation*}%
and $u^{k+1}$ is the unique solution of 
\begin{equation*}
\frac{d}{dt}u+\widetilde{A}u=\widetilde{F}(u^{k})
\end{equation*}

By induction one can easily show that the sequence $\left( u^{k}\right) $ is
monotone non-decreasing and $u^{k}\in K,$ $\forall k\geq 1.$ Let us show the
first step, that is $u^{1}\geq u_{1}.$

First using the assumptions on $A$ we have the existence of a compact
positive semi-group $S(t),$ generated by $\widetilde{A}$ (see \cite{Paz}).
So we have 
\begin{eqnarray*}
u_{1} &=&S(t)u_{1}(0)+\int_{0}^{t}S(t-s)\left( \frac{d}{dt}u_{1}+\widetilde{A%
}u_{1}\right) \\
&\leq &S(t)u_{0}+\int_{0}^{t}S(t-s)\widetilde{F}(u_{1}(s))ds \\
&\leq &u^{1}
\end{eqnarray*}%
Now since $A$ has compact resolvent and $K$ is bounded we can extract for
fixed time a subsequence that we will call also $\left( u^{k}\right) $ such
that $S(t)u^{k}(s)$ converges to $S(t)u(s),$ thus by writing 
\begin{equation*}
u^{k+1}=S(t)u_{0}+\int_{0}^{t}S(t-s)\widetilde{F}(u^{k}),
\end{equation*}%
one can see that $u$ satisfies 
\begin{equation*}
u=S(t)u_{0}+\int_{0}^{t}S(t-s)\widetilde{F}(u)
\end{equation*}%
And this gives a positive solution to (\ref{Prob H}).

Now notice that $u(t)$ is bounded in $D(A)$ thus there exists a sequence $%
\left( t_{k}\right) _{k}$ going to infinity such that $u(t_{k})$ converges
to some $\widetilde{u},$ and in fact the convergence occurs in $D(A)$. Thus
knowing that $\lim \int_{0}^{t}S(t-s)xds=\left( -A\right) ^{-1}x$ we get by
passing to the limit that $\widetilde{u}$ is a solution of the steady-state
problem.
\end{proof}

Now we will consider a problem of the form 
\begin{equation}
\left\{ 
\begin{array}{c}
u_{t}+P_{g,\Psi }u=f(x,u) \\ 
u(0)=u_{0}%
\end{array}%
\right.  \label{prob es}
\end{equation}%
where $P$ is the Paneitz-Branson operator and $f:M\times \mathbb{R}_{+}^{\ast }\longrightarrow \mathbb{R}$ is $C^{1}$ such that $\lim_{x\longrightarrow +\infty }f(x,t)=-\infty $
uniformly on $x$ and $\lim_{x\longrightarrow 0}f(x,t)=+\infty $ uniformly on 
$x,$ Then for every $u_{0}$ smooth and positive, there exist a positive
solution to (\ref{prob es}) moreover there exist a sequence $t_{k}$ going to
infinity such that $u(t_{k})\longrightarrow \widetilde{u}$ a solution of the
steady-state problem.

One also can write the problem as an integral equation using the positivity
preserving flow like in \cite{Ras}.

\begin{cor}
Take $A$ and $B$ two positive functions defined on $M$ and consider the
singular problem :%
\begin{equation}
\left\{ 
\begin{array}{c}
Pu=\frac{A(x)}{u^{p}}-B(x)u^{q} \\ 
u>0%
\end{array}%
\right.   \label{Prob A}
\end{equation}%
then using Theorem (3.2) we have the existence of a solution more than that,
it is the unique solution.
\end{cor}

Remark that in this case we can take $q\geq 2^{\sharp }-1$ since we do not
need the compact or continuous embedding in $L^{p}$ spaces.

Now,if we suppose that $B$ is just non-negative. we can show indeed that
even in that case we have a solution.

\begin{cor}
Take $A$ $>0$ and $B\geq 0,$ two smooth functions defined on $M$ and consider
the singular problem :%
\begin{equation*}
\left\{ 
\begin{array}{c}
P_{g,\psi }u=\frac{A(x)}{u^{p}}-B(x)u^{q} \\ 
u>0%
\end{array}%
\right.
\end{equation*}%
where $q\leq 2^{\sharp }-1,$ then it has a unique solution.
\end{cor}

Let $u_{\varepsilon }$ be the solution obtained by Corollary (3.3), of

\begin{equation*}
\left\{ 
\begin{array}{c}
P_{g,\psi }u=\frac{A(x)}{u^{p}}-B_{\varepsilon }(x)u^{q} \\ 
u>0%
\end{array}%
\right.
\end{equation*}

\bigskip where $B_{\varepsilon }=B+\varepsilon .$ First remark that $%
u_{\varepsilon }$ is uniformly bounded from bellow (it is by construction of
the sub and super solution in the proof of Theorem (3.2)).

So 
\begin{eqnarray*}
\int_{M}uP_{g,\psi }u &=&\int_{M}\frac{A(x)}{u_{\varepsilon }^{p-1}}%
-\int_{M}\left( B(x)+\varepsilon \right) u_{\varepsilon }^{q+1} \\
&\leq &\int_{M}\frac{A(x)}{\delta }=C
\end{eqnarray*}%
where $\delta $ is the uniform lower bound of $u_{\varepsilon }.$ Therefore $%
\left( u_{\varepsilon }\right) _{\varepsilon }$ is bounded in $H^{2}\left(
M\right) $ and if $q+1\leq 2^{\sharp },$ we have $u_{\varepsilon
}\longrightarrow u$ in $L^{2}$ and weakly in $H^{2}\left( M\right) $ and $%
L^{2^{\sharp }}(M).$

So if we take $\varphi \in C^{\infty }(M),$ we have a weak solution which we
can show using the regularity theory that is is indeed smooth. 
\begin{equation*}
\int_{M}\varphi P_{g,\psi }u_{\varepsilon }=\int_{M}\frac{A(x)}{%
u_{\varepsilon }^{p-1}}\varphi -\int_{M}\left( B(x)+\varepsilon \right)
u_{\varepsilon }^{q}\varphi
\end{equation*}%
so by letting $\varepsilon \longrightarrow 0$ we get that 
\begin{equation*}
\int_{M}\varphi P_{g,\psi }u=\int_{M}\frac{A(x)}{u^{p-1}}\varphi
-\int_{M}B(x)u^{q}\varphi
\end{equation*}%
so $u$ is a weak solution and using elliptic regularity we get the fact that
it is indeed a smooth one.

For the uniqueness, if we consider two smooth positive solutions $u$ and $v$
of then $w=u-v$ satisfies :%
\begin{eqnarray*}
P_{g,\psi }w &=&\frac{A(x)}{u^{p}}-\frac{A(x)}{v^{p}}+B_{\varepsilon
}(x)v^{q}-B_{\varepsilon }(x)u^{q} \\
&=&-C(x)(u-v)=-C(x)w
\end{eqnarray*}%
where $C(x)$ is a non-negative function that we get from the mean value
theorem, therefore using the maximum principle we get the desired result.

As an improvement of the previous result we have :

\begin{teo}
Assume that $B^{+}$ is non-zero then problem (\ref{Prob A}) has at least one
positive solution if the following inequality is satisfied 
\begin{equation}
\max_{M}\left( A^{\frac{q-1}{p+q}}B_{-}^{\frac{p+1}{p+q}}\varphi _{1}^{q%
\frac{p+1}{p+q}-p\frac{q-1}{p+q}-1}\right) \leq \frac{\lambda _{1}}{\left( 
\frac{q-1}{p+1}\right) ^{\frac{p+1}{p+q}}+\left( \frac{p+1}{q-1}\right) ^{%
\frac{q-1}{p+q}}},  \label{ineq}
\end{equation}%
where $\lambda _{1}$ and $\varphi _{1}$ are the first eigenvalue and
eigenfunction of $P_{g,\psi },$ respectively.
\end{teo}

first let \underline{$u$} be a solution of 
\begin{equation*}
\left\{ 
\begin{array}{c}
P_{g,\psi }u=\frac{A(x)}{u^{p}}-B^{+}(x)u^{q} \\ 
u>0%
\end{array}%
\right. 
\end{equation*}%
In fact since we are going to use this process another time let us give the
picture and the idea behind :

Consider a convex function positive $f:\mathbb{R}^{+}\longrightarrow \mathbb{R}$ and, so for it to intersect a line $L$ passing through the origin its
slope should be greater than the one of the unique tangent to the graph of
passing through the origin as shown in the following figure : \\
\begin{center}
\includegraphics[width=80mm,height=60mm]{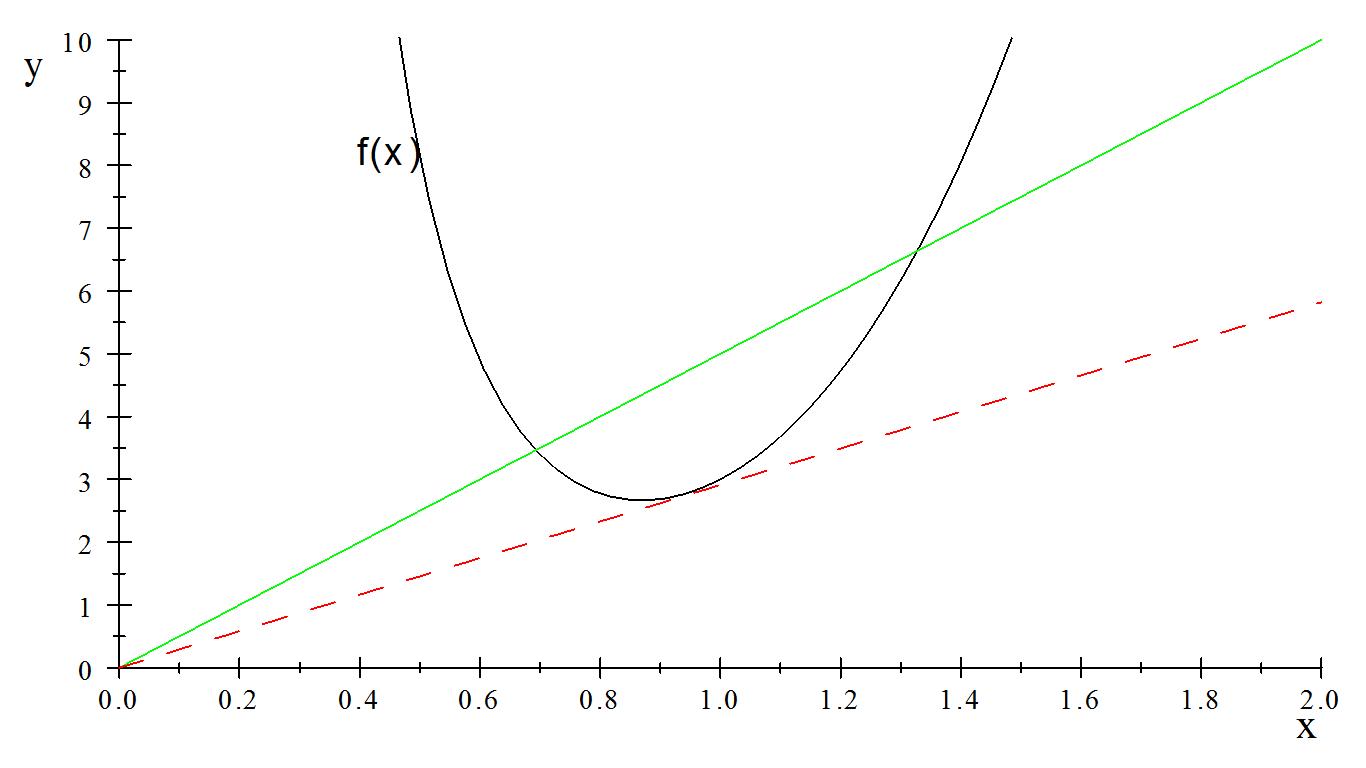}
\end{center}
and the condition to find the slope at zero is by solving 
\begin{equation}
\frac{f(t)}{t}=f^{\prime }(t).  \label{conv}
\end{equation}

So if we take $\varphi _{1}$ the first eigenfunction of $P$, we get 
\begin{eqnarray*}
tP\varphi _{1}-\frac{A(x)}{t^{p}\varphi ^{p}}+B(x)t^{q}\varphi _{1}^{q}
&=&t\lambda _{1}\varphi _{1}-\frac{A(x)}{t^{p}\varphi _{1}^{p}}%
+B(x)t^{q}\varphi _{1}^{q} \\
&\geq &t\lambda _{1}\varphi _{1}-\frac{A(x)}{t^{p}\varphi _{1}^{p}}%
-B^{-}(x)t^{q}\varphi _{1}^{q}.
\end{eqnarray*}%
And here we can see that in fact we are comparing $t\longrightarrow t\lambda
_{1}\varphi _{1}$ and $t\longrightarrow \frac{A(x)}{t^{p}\varphi _{1}^{p}}%
+B^{-}(x)t^{q}\varphi _{1}^{q}$ which is convex, thus using the previous
remark we can see that the inequality (\ref{ineq}) insures that we are in
the same situation as \ref{Fig1} and thus there exist $t_{0}>0$ such that $%
t_{0}\varphi _{1}$ is a super-solution to (\ref{Prob A}) therefore, using
the classical monotone iteration method we get a positive solution.

\subsection{Further investigations and existence results}

Here we investigate the case where $B<0.$ The coercivity assumption implies
that 
\begin{equation*}
\left\Vert u\right\Vert _{\psi }=\left( \int_{M}uP_{g,\psi }u\right) ^{\frac{%
1}{2}},
\end{equation*}%
defines a norm equivalent to the $H^{2}(M)$ norm. So we will use that norm
instead of the usual one. Also we take $S_{\psi }$ the best Sobolev constant
with respect to this norm, that is $S_{\psi }$ is the est constant satisfying 
\begin{equation*}
\left\Vert u\right\Vert _{L^{2^{\sharp }}}^{2}S_{\psi }\leq \left\Vert
u\right\Vert _{\psi }^{2}
\end{equation*}

Remark that for $B<0$ this condition still work, but let us try to find
another condition that works in a weaker setting. We will rewrite the
problem as 
\begin{equation}
\left\{ 
\begin{array}{c}
P_{g,\psi }u=\frac{A(x)}{u^{p}}+B(x)u^{q} \\ 
u>0%
\end{array}%
\right. ,  \label{Prob B}
\end{equation}%
and $B$ here is taken to be positive (in fact one get a similar result if $B$
has a negative part up to a small modification to the assumption in the
following theorem).

For the regularity issues we refer to \cite{Car} and \cite{Car2}, there one
can find the necessary regularity and bootstrapping argument to deal with it.

\begin{teo}
Assume that $P$ is strongly positive (that is it satisfies the strong
maximum principle). If there exist a function $\varphi >0$ in $H^{2}\left(
M\right) $ such that 
\begin{equation}
\left\Vert \varphi \right\Vert _{\psi }^{(p-1)}\left\Vert B\right\Vert _{L^{%
\frac{2^{\sharp }}{2^{\sharp }-q-1}}}^{\frac{p+1}{q-1}}\left( \int_{M}\frac{A%
}{\varphi ^{p-1}}\right) <C  \label{Cond}
\end{equation}%
then problem (\ref{Prob B}) has at least one positive smooth solution.
\end{teo}

In fact we will compute an exact value of $C,$ That is 
\begin{equation}
C=S_{\psi }^{-\frac{\left( q+1\right) \left( p+2q+1\right) }{2(q-1)}}\left( 
\frac{\left( q-1\right) \left( p-1\right) }{2}\right)   \label{Con}
\end{equation}%
Befor Starting the proof let us state the following lemma which appears to
be helpful in our situation.

\begin{lemma}
Let $E,$ $E_{1}$, $E_{2}$ be three $C^{1}$ functional on a Banach space $X.$
Assume that $E_{1}(0)=0$ and $\lim E(t\varphi )=-\infty $. and $E_{2}\geq 0.$
Then If $E_{1}$ has the montainpass geometry around zero, (that is there
exist $r>0$ such that $\delta =\inf_{\partial B(0,r)}E_{1}>0$) and there
exist $u\in B(0,r)$ such that $E_{2}(u)<\delta $, the functional $E$ has a
Palais-Smale sequence.
\end{lemma}

\begin{proof}[Proof of Lemma]
Here is is easy to see that if we concider the set 
\begin{equation*}
\Gamma =\left\{ \gamma :\left[ 0,1\right] \longrightarrow X\text{ such that }%
\gamma (0)=u\text{ and }\gamma (1)=t\varphi \right\} 
\end{equation*}%
then we get a Palais-Smale sequence at the level%
\begin{equation*}
c=\inf_{\gamma \in \Gamma }\max E(\gamma (\left[ 0,1\right] ))
\end{equation*}%
Since each curve crosses $\partial B(0,r)$, then $c>\max (E(u),E(t\varphi )),
$ and thus we have a mountain pass geometry.
\end{proof}

\begin{proof}
In fact let $\varphi $ be a positive function such that $\left\Vert \varphi
\right\Vert _{\psi }=1$ and the energy functional 
\begin{equation*}
E(u)=\frac{1}{2}\left\Vert u\right\Vert _{\psi }^{2}+\frac{1}{p-1}\int_{M}%
\frac{A}{\left( \varepsilon +u^{2}\right) ^{\frac{p-1}{2}}}-\frac{1}{q+1}%
\int_{M}Bu^{q+1}.
\end{equation*}%
Clearly the functional $E_{1}$ defined by 
\begin{equation*}
E_{1}(u)=\frac{1}{2}\left\Vert u\right\Vert _{\psi }^{2}-\frac{1}{q+1}%
\int_{M}Bu^{q+1}
\end{equation*}%
has the mountain pass geometry and in fact if $r_{0}=\left\Vert B\right\Vert
_{L^{s}}^{\frac{-1}{q-1}}S_{\psi }^{-\frac{q+1}{2(q-1)}}$ then 
\begin{equation*}
\inf_{u\in \partial B(0,r_{0})}E_{1}(u)=\left\Vert B\right\Vert _{L^{s}}^{%
\frac{-2}{q-1}}S^{-\frac{q+1}{q-1}}\left( \frac{q-1}{2}\right) 
\end{equation*}%
And therefore the inequality (\ref{Cond}) is exactly saying that that $%
t_{0}\varphi $ satisfies the assumption of Lemma (3.7) for $t_{0}<r_{0}$ and
thus we have the existence of $t_{0}<r_{0}<t_{2}$ such that, 
\begin{equation*}
\max (E(t_{0}\varphi ),E(t_{2}\varphi ))<E(r_{0}\varphi )
\end{equation*}%
And in fact we can apply the lemma for the following approximated energy
functional%
\begin{equation*}
E_{\varepsilon }(u)=\frac{1}{2}\left\Vert u\right\Vert _{\psi }^{2}+\frac{1}{%
p-1}\int_{M}\frac{A}{\left( \varepsilon +\left( u^{+}\right) ^{2}\right) ^{%
\frac{p-1}{2}}}-\frac{1}{q+1}\int_{M}B\left( u^{+}\right) ^{q+1}
\end{equation*}%
for $\varepsilon >0$ and small. Remark that we have uniform convergence of $%
t\longrightarrow E_{\varepsilon }(t\varphi )$ to $t\longrightarrow
E(t\varphi ),$ on every compact of $\mathbb{R}.$ Therefore there exist $\varepsilon _{0}>0$ such that for every $%
0<\varepsilon \leq \varepsilon _{0}$ one have
\begin{equation}
\max (E_{\varepsilon _{0}}(t_{0}\varphi ),E_{\varepsilon _{0}}(t_{2}\varphi
))\leq \max (E_{\varepsilon }(t_{0}\varphi ),E_{\varepsilon }(t_{2}\varphi
))<E_{\varepsilon }(r_{0}\varphi )\leq E(r_{0}\varphi ).  \label{crit bound}
\end{equation}%
Therefore if we take 
\begin{equation*}
\Gamma =\left\{ \gamma :\left[ 0,1\right] \longrightarrow H^{2}\left(
M\right) \text{ such that }\gamma \left( 0\right) =t_{0}\varphi ;\text{ }%
\gamma \left( 1\right) =t_{2}\varphi \right\} 
\end{equation*}%
we have a Palais-smale sequence at the level 
\begin{equation*}
c_{\varepsilon }=\inf_{\gamma \in \Gamma }\max_{u\in \gamma \left( \left[ 0,1%
\right] \right) }E_{\varepsilon }(u)
\end{equation*}%
notice that from (\ref{crit bound}) 
\begin{equation*}
\inf_{u\in \partial B(0,r_{0})}E_{1}(u)<c_{\varepsilon }<E(r_{0}\varphi )
\end{equation*}%
and thus $c_{\varepsilon }$ is uniformly bounded. Let us call that
Palais-Smale sequence $(u_{k}^{\varepsilon })_{k}$, it satisfies then 
\begin{equation*}
E_{\varepsilon }(u_{k}^{\varepsilon })\longrightarrow c_{\varepsilon }\text{
and }E_{\varepsilon }^{\prime }(u_{k}^{\varepsilon })\longrightarrow 0\text{
in }H^{-2}\text{ as }k\longrightarrow \infty .
\end{equation*}%
Thus The following holds 
\begin{eqnarray}
O(\left\Vert u_{k}^{\varepsilon }\right\Vert _{\psi }) &=&\left( q+1\right)
E_{\varepsilon }(u_{k}^{\varepsilon })-\left\langle E_{\varepsilon }^{\prime
}(u_{k}^{\varepsilon }),u_{k}^{\varepsilon }\right\rangle   \label{eng est}
\\
&=&\frac{\left( q-1\right) }{2}\left\Vert u_{k}^{\varepsilon }\right\Vert
_{\psi }^{2}+\left( \frac{q+1}{p-1}-1\right) \int \frac{A}{\left(
\varepsilon +\left( u_{k}^{\varepsilon +}\right) ^{2}\right) ^{\frac{p-1}{2}}%
}+\varepsilon \int \frac{A}{\left( \varepsilon +\left( u_{k}^{\varepsilon
+}\right) ^{2}\right) ^{\frac{p+1}{2}}}
\end{eqnarray}%
Therefor 
\begin{equation*}
\left\Vert u_{\varepsilon }^{k}\right\Vert _{\psi }=O(1),
\end{equation*}%
which implies the boundedness of $\left( u_{k}^{\varepsilon }\right) _{k}$
in $H^{2}\left( M\right) $ and thus the existwwence of $u_{\varepsilon }\in
H^{2}\left( M\right) $ such that 
\begin{equation*}
\left\{ 
\begin{array}{c}
u_{k}^{\varepsilon }\rightharpoonup u_{\varepsilon }\text{ weakly in }H^{2}
\\ 
u_{k}^{\varepsilon }\longrightarrow u_{\varepsilon }\text{ strongly in }L^{2}
\\ 
u_{k}^{\varepsilon }\rightharpoonup u_{\varepsilon }\text{ weakly in }%
L^{2^{\sharp }}%
\end{array}%
\right. ,
\end{equation*}%
so take $\eta \in C^{\infty }(M)$, the previous assertion gives that 
\begin{equation*}
\int_{M}\eta P_{g,\psi }u_{\varepsilon }=\int_{M}\frac{Au_{\varepsilon
}^{+}\eta }{\left( \varepsilon +\left( u_{\varepsilon }^{+}\right)
^{2}\right) ^{\frac{p-1}{2}}}+\int_{M}Bu_{\varepsilon }^{q}\eta 
\end{equation*}%
thus $u_{\varepsilon }$ is a weak solution to the problem 
\begin{equation}
P_{g,\psi }u_{\varepsilon }=\frac{Au_{\varepsilon }^{+}}{\left( \varepsilon
+\left( u_{\varepsilon }^{+}\right) ^{2}\right) ^{\frac{p+1}{2}}}%
+Bu_{\varepsilon }^{q}  \label{appr}
\end{equation}%
hence $u_{\varepsilon }$ is smooth and positive.

First, assume that $\left( \frac{q+1}{p-1}-1\right) >0$, then 
\begin{equation*}
\left( q+1\right) E_{\varepsilon }(u_{\varepsilon })-\left\langle
E_{\varepsilon }^{\prime }(u_{\varepsilon }),u_{\varepsilon }\right\rangle
=\left( q+1)\right) c_{\varepsilon }
\end{equation*}%
therefore from (\ref{eng est}) we get 
\begin{eqnarray}
\int_{M}\frac{A}{\left( \varepsilon +\left( u_{\varepsilon }\right)
^{2}\right) ^{\frac{p+1}{2}}} &<&C_{1}  \label{est} \\
\left\Vert u_{\varepsilon }\right\Vert _{H^{2}} &<&C_{2}  \notag
\end{eqnarray}%
where $C_{1},$ $C_{2}$ are constants independant of $\varepsilon .$ Thus we
can extract a subsequence of $\left( u_{\varepsilon }\right) _{\varepsilon }$
that we will call $\left( u_{\varepsilon _{k}}\right) _{\varepsilon _{k}}$
so that 
\begin{equation*}
\left\{ 
\begin{array}{c}
u_{\varepsilon _{k}}\rightharpoonup u\text{ weakly in }H^{2}(M) \\ 
u_{\varepsilon _{k}}\longrightarrow u\text{ strongly in }L^{s}(M)\text{ for }%
1<s<2^{\sharp } \\ 
u_{\varepsilon _{k}}\longrightarrow u\text{ a.e on }M.%
\end{array}%
\right.
\end{equation*}%
Thus using Fatou's lemma in (\ref{est}) we get 
\begin{equation}
\int_{M}\frac{1}{u^{p+1}}<C_{1}.  \label{int est}
\end{equation}%
Assume now that there exist $x_{k}\longrightarrow \overline{x}$ such that $%
u_{\varepsilon _{k}}(x_{k})\longrightarrow 0$. Then using the integral
representation we get 
\begin{equation*}
u_{\varepsilon _{k}}(x_{k})\geq \int_{M}G(x_{k},y)B\left( y\right)
u_{\varepsilon _{k}}^{q}(y)dy
\end{equation*}%
where $G$ is the Green's function of the operator $P_{g,\psi }.$ Taking $%
k\longrightarrow 0$ we get 
\begin{equation*}
\int_{M}G(\overline{x},y)B\left( y\right) u^{q}(y)dy=0,
\end{equation*}%
thus $u=0$ which is impossible because of (\ref{int est}), therfor $%
u_{\varepsilon }$ is uniformly bounded from below.

So now we can pass to the weak limit in (\ref{appr}) to get. 
\begin{equation*}
\int_{M}\eta P_{g,\psi }u=\int_{M}\frac{\eta A}{u^{p}}+\int_{M}Bu^{q}\eta ,%
\text{ for every }\eta \in C^{\infty }(M),
\end{equation*}%
hence, since $u$ is positively bounded from below, we get a smooth positive
solution to 
\begin{equation*}
P_{g,\psi }u=\frac{A}{u^{p}}+Bu^{q}.
\end{equation*}%
If $p-1=q+1$ (and that is the case of the Lichnerowicz Equation), to find a
uniform bound on $\int_{M}\frac{A}{\left( \varepsilon +\left( u_{\varepsilon
}\right) ^{2}\right) ^{\frac{p+1}{2}}},$ we use the fact that $\left\Vert
u_{\varepsilon }\right\Vert _{\psi }$ is uniformly bounded, and Sobolev
embedding to get a uniform bound on $\int_{M}Bu_{\varepsilon }^{q+1}$ and
thus we get the desired bound.
\end{proof}

\begin{cor}
Under the assumption of the previous theorem, we have the existence of
another positive solution
\end{cor}

\begin{proof}
If we take a look at the inequality (\ref{Cond}) we notice that it is open,
that is iw we pertube $B$ a small perturbation, we still get the same
existence result. So let us call $u_{B}$ the solution corresponding to $%
B.$ Then using a comparison principle, we get $u_{B-\varepsilon
}<u_{B+\varepsilon }$ and they are a pair of sub and super-solution to the
problem (\ref{Prob B}), therefore we have the existence of a solution $%
\widetilde{u}$ to the problem, and to guaranty that $u_{B}\neq \widetilde{u}$
we use a degree theory argument since every positive smooth solution is in
the set $A=\left\{ u\in C^{4,\alpha }\left( M\right) ;\text{ }\frac{1}{C}%
<u<C\right\} $ for $C>0$ large enough and uniform.
\end{proof}

\begin{cor}
There exist a constant $C=C(n,M,Q_{\psi })>0$ such that, if $P_{g,\psi }$ is
strongly positive and 
\begin{equation*}
\left( \max B\right) ^{\frac{3n-4}{8}}\int_{M}A<C
\end{equation*}%
the Paneitz-Lichnerowicz Problem admits at least one positive solution.
\end{cor}

For the proof of this corollary, we just take $\varphi =1$ in (\ref{Cond})$,$
and the Sobolev embedding 
\begin{equation*}
H^{2}(M)\hookrightarrow L^{2^{\sharp }}\left( M\right) .
\end{equation*}

\section{Non existence Result}

\begin{teo}
\bigskip Assume that $A,B\geq 0$, then if 
\begin{equation*}
\left( \int_{M}A^{\frac{q}{p+q}}B^{\frac{p}{q+p}}\right) ^{\frac{\left(
p+q\right) \left( q-3\right) }{q\left( p+q-2\right) }}\left( 
\begin{array}{c}
\left( \frac{q-1}{p+1}\right) ^{\frac{1-q}{p+q-2}}\left( \int_{M}A^{\frac{q}{%
p+q}}B^{\frac{p}{q+p}}\right) ^{\frac{\left( p+q\right) 2}{q\left(
p+q-2\right) }} \\ 
+\left( \frac{q-1}{p+1}\right) ^{\frac{p+1}{p+q-2}}%
\end{array}%
\right) >\left( \int_{M}\left( Q_{\psi }^{+}\right) ^{\frac{q}{q-1}}B^{-%
\frac{1}{q-1}}\right) ^{\frac{q-1}{q}},
\end{equation*}%
then the problem does not posses any positive smooth solution, where $%
Q_{\psi }=Q-\left\vert \nabla \psi \right\vert _{g}^{2}$
\end{teo}

\begin{proof}
Let $u$ \ be a positive solution, then the follwong holds :%
\begin{equation*}
\int_{M}\frac{A}{u^{p}}+\int_{M}Bu^{q}=\int_{M}Q_{\psi }u
\end{equation*}%
using the fact that%
\begin{equation*}
\int_{M}Q_{\psi }u\leq \left( \int_{M}\left( Q_{\psi }^{+}\right) ^{\frac{q}{%
q-1}}B^{-\frac{1}{q-1}}\right) ^{\frac{q-1}{q}}\left( \int_{M}Bu^{q}\right)
^{\frac{1}{q}}
\end{equation*}%
Also 
\begin{equation*}
\int_{M}A^{\frac{q}{p+q}}B^{\frac{p}{q+p}}\leq \left( \int_{M}\frac{A}{u^{p}}%
\right) ^{\frac{q}{p+q}}\left( \int_{M}Bu^{q}\right) ^{\frac{p}{q+p}}
\end{equation*}%
therefore if we set $X=\int Bu^{q},$ one gets%
\begin{equation*}
X+\left( \left( \int_{M}A^{\frac{q}{p+q}}B^{\frac{p}{q+p}}\right) \right) ^{%
\frac{p+q}{q}}X^{-\frac{p}{q}}\leq \left( \int_{M}\left( Q_{\psi
}^{+}\right) ^{\frac{q}{q-1}}B^{-\frac{1}{q-1}}\right) ^{\frac{q-1}{q}}X^{%
\frac{1}{q}}
\end{equation*}%
Which is equivalent to say that 
\begin{equation*}
X^{1-\frac{1}{q}}+\left( \int_{M}A^{\frac{q}{p+q}}B^{\frac{p}{q+p}}\right) ^{%
\frac{p+q}{q}}X^{-\frac{p+1}{q}}\leq \left( \int_{M}\left( Q_{\psi
}^{+}\right) ^{\frac{q}{q-1}}B^{-\frac{1}{q-1}}\right) ^{\frac{q-1}{q}}
\end{equation*}%
Therefore if 
\begin{equation*}
\left( \int_{M}A^{\frac{q}{p+q}}B^{\frac{p}{q+p}}\right) ^{\frac{\left(
p+q\right) \left( q-3\right) }{q\left( p+q-2\right) }}\left( 
\begin{array}{c}
\left( \frac{q-1}{p+1}\right) ^{\frac{1-q}{p+q-2}}\left( \int_{M}A^{\frac{q}{%
p+q}}B^{\frac{p}{q+p}}\right) ^{\frac{\left( p+q\right) 2}{q\left(
p+q-2\right) }} \\ 
+\left( \frac{q-1}{p+1}\right) ^{\frac{p+1}{p+q-2}}%
\end{array}%
\right) >\left( \int_{M}\left( Q_{\psi }^{+}\right) ^{\frac{q}{q-1}}B^{-%
\frac{1}{q-1}}\right) ^{\frac{q-1}{q}}
\end{equation*}%
then there is no smooth positive solution to the problem.
\end{proof}

\section{Conclusion}

As a conclusion of the previous existence and non-existence result, we can
set for the sake of simplicity, $A=1,$ $B=\lambda \in \mathbb{R},$ and we get the following corollary if we consider the following problem  
\begin{equation}
P_{g,\psi }u=\frac{1}{u^{p}}+\lambda u^{q}.  \label{Prob conc}
\end{equation}

\begin{cor}
If $P_{g,\psi }$ is strongly positive, then there exist a constant $\lambda
^{\ast }>0$ such that 

i)Problem has no positive smooth solution if $\lambda >\lambda ^{\ast }.$

ii)Problem has at least one positive solution if $\lambda <\lambda ^{\ast }.$

Moreover we have the following estimate 
\begin{equation*}
\left( Vol(M)^{-\frac{2^{\sharp }-q-1}{2^{\sharp }}}C\left( \frac{n-4}{2}%
\int_{M}Q_{\psi }\right) ^{-(p-1)}\right) ^{\frac{q-1}{p+1}}<\lambda ^{\ast
}<Vol(M)^{-\frac{(p+q)(q-1)}{pq+q-2}}\left( \frac{q-1}{p+1}\right) ^{\frac{%
q\left( q-1\right) }{pp+q-2}}\left\Vert Q_{\psi }\right\Vert _{\frac{q}{q-1}%
}^{\frac{q(p+q-2)}{pq+q-2}},
\end{equation*}%
where $C$ is the constant (\ref{Con}).
\end{cor}


\begin{thebibliography}{10}
\bibitem[1]{Am} A. Ambrosetti and P. Rabinowitz, Dual variational methods in
critical point theory and applications, J. Funct. Anal. 14, 1973, 349--381

\bibitem[2]{Ben} Ben Ayed, M. ; El Mehdi, K., The Paneitz curvature problem
on lower-dimensional spheres. Ann. Global Anal. Geom. 31 (2007), no. 1,
1--36.

\bibitem[3]{Bra1} T. Branson; Differential operators canonically associated
to a conformal structure, Math.Scand. 57 (1985), 293-345.

\bibitem[4]{Bra2} T.P. Branson and A.R. Gover, Origins, applications and
generalisations of the Q-curvature, Acta Appl. Math. 102 (2008), 131--146.

\bibitem[5]{Car} D. Caraffa, \'{E}quations elliptiques du quatri\`{e}me
ordre avec exposants critiques sur les vari\'{e}t\'{e}s riemanniennes
compactes, J. Math. Pures Appl. (9) 80 (2001), no. 9, 941--960.

\bibitem[5]{Car2} D. Caraffa, \'{E}tude des probl\`{e}mes elliptiques non lin%
\'{e}aires du quatri\`{e}me ordre avec exposants critiques sur les vari\'{e}t%
\'{e}s riemanniennes compactes, J. Math. Pures Appl. (9) 83 (2004), no. 1,
115--136.

\bibitem[6]{Chang} S.Y.A Chang, P.C. Yang, On a fourth order curvature
invariant, Comp. Math. 237, Spectral Problems in Geometry and Arithmetic,
Ed: T. Branson, AMS, 1999, 9-28.

\bibitem[7]{Cho1} Y. Choquet-Bruhat, Results and open problems in
mathematical general relativity. Milan J. Math. 75 (2007), 273--289.

\bibitem[8]{Cho2} Y. Choquet-Bruhat and R. Geroch, Global aspects of the
Cauchy problem in general relativity, Comm. Math. Phys. 14 (1969), 329-335

\bibitem[9]{Cho3} Y. Choquet-Bruhat; J. Isenberg; D. Pollack, Applications
of theorems of Jean Leray to the Einstein-scalar field equations. J. Fixed
Point Theory Appl. 1 (2007), no. 1, 31--46

\bibitem[10]{Cho 4} Y. Choquet-Bruhat; J. Isenberg; D. Pollack, The
constraint equations for the Einstein-scalar field system on compact
manifolds. Classical Quantum Gravity 24 (2007), no. 4, 809--828.

\bibitem[11]{Dru} O. Druet; E. Hebey, Stability and instability for
Einstein-scalar field Lichnerowicz equations on compact Riemannian
manifolds. Math. Z. 263 (2009), no. 1, 33--67

\bibitem[12]{Dja} Z. Djadli, E. Hebey, M. Ledoux, Paneitz type operators and
applications, Duke Math. J. 104, 2000, 129--169

\bibitem[13]{Dja2} D. Djadli; A. Malchiodi, Existence of conformal metrics
with constant \$Q\$-curvature. Ann. of Math. (2) 168 (2008), no. 3, 813--858.

\bibitem[14]{Dja3} Z. Djadli; A. Malchiodi;M. O. Ahmedou, Prescribing a
fourth order conformal invariant on the standard sphere. II. Blow up
analysis and applications. Ann. Sc. Norm. Super. Pisa Cl. Sci. (5) 1 (2002),
no. 2, 387--434.

\bibitem[15]{Heb1} E. Hebey; F. Pacard; D. Pollack ,A variational analysis
of Einstein-scalar field Lichnerowicz equations on compact Riemannian
manifolds. Comm. Math. Phys. 278 (2008), no. 1, 117--132.

\bibitem[16]{Heb2} Hebey, E., and Robert, F., Coercivity and Struwe's
compactness for Paneitz type operators with constant coefficients, Calc.
Var. Partial Differential Equations, 13, 2001, 491-517.

\bibitem[17]{Heb3} E. Hebey; F. Robert; Y. Wen, Compactness and global
estimates for a fourth order equation of critical Sobolev growth arising
from conformal geometry. Commun. Contemp. Math. 8 (2006), no. 1, 9--65.

\bibitem[18]{Lee} J.M. Lee; T .H. Parker, The Yamabe problem. Bull. Amer.
Math. Soc. 17,(1987),no. 1, 37--91

\bibitem[19]{Li} Li Ma, Yuhua Sun, Heat flow method to Lichnerowicz type
equation on closed manifolds

\bibitem[20]{Pan} S. Paneitz; A quartic conformally covariant di erential
operator for arbitrary pseudo-Riemannian manifolds, Preprint, 1983.

\bibitem[21]{Paz} Pazy, A. Semigroups of linear operators and applications
to partial differential equations. Applied Mathematical Sciences, 44.
Springer-Verlag, New York, 1983.

\bibitem[22]{Ras} D. Raske, A Fourth-Order Positivity Preserving Geometric
Flow, arXiv:math/0608146v3.

\bibitem[23]{Ras2} D. Raske, Prescription of Q-curvature on closed
Riemannian manifolds, arXiv:0806.3790v3.

\bibitem[24]{Yang} Xu, X., and Yang, P., Positivity of Paneitz operators,
Discrete and Continuous Dynamical Systems 7, 2001, 329-342.
\end{thebibliography}
\end{document}